\documentclass[11pt,reqno]{amsart}

\usepackage{amssymb,amsmath}
\usepackage[pagebackref,hypertexnames=false, colorlinks, citecolor=red, linkcolor=red]{hyperref} 
\usepackage[backrefs]{amsrefs}

\newcommand{\be}{\begin{eqnarray}}
\newcommand{\ee}{\end{eqnarray}}

\newcommand{\e}{{\varepsilon}}

\newcommand{\al}{\alpha}
\newcommand{\Dk}{{\mathcal D}}

\newcommand{\La}{{\Lambda}}

\newcommand{\cz}{Calder\'on-Zygmund}
\newcommand{\vf}{\varphi}
\newcommand{\pd}{\partial}

\newcommand{\dist}{\operatorname{dist}}
\newcommand{\ci}[1]{_{{}_{\scriptstyle{#1}}}}

\newcommand{\supp}{\operatorname{supp}}

\newtheorem{theorem}{Theorem}
\newtheorem{lemma}[theorem]{Lemma}

\theoremstyle{definition}
\newtheorem{defi}[theorem]{Definition}
\theoremstyle{remark}
\newtheorem{remark}[theorem]{Remark}
\numberwithin{equation}{section}\input epsf.sty

\begin{document}
\thispagestyle{empty}

\title[Bergman-type Singular Integral Operators on Metric Spaces]{{Bergman-type Singular Integral Operators on Metric Spaces}}

\author[A. Volberg]{Alexander Volberg$^{\dagger}$}
\address{Alexander Volberg, Department of  Mathematics\\ Michigan State University\\ East Lansing, MI USA 48824}
\email{volberg@math.msu.edu}
\address{Alexander Volberg, Department of Mathematics\\ University of Edinburgh\\ James Clerk Maxwell Building\\ The King's Buildings\\ Mayfield Road\\ Edinburgh Scotland EH9 3JZ}
\email{a.volberg@ed.ac.uk}
\thanks{$\dagger$ Research supported in part by a National Science Foundation DMS grant.}

\author[B. D. Wick]{Brett D.~Wick$^{\ddagger}$} 
\address{Brett D. Wick, School of Mathematics, Georgia Institute of Technology\\ 686
  Cherry Street\\ Atlanta, GA 30332-1060 USA}
\email{wick@math.gatech.edu} 
\urladdr{http://people.math.gatech.edu/~bwick6/} 
\thanks{$\ddagger$ The second author is supported by National Science Foundation CAREER Award DMS\# 0955432 and an Alexander von Humboldt Fellowship.}

\begin{abstract}
In this paper we study ``Bergman-type'' singular integral operators on Ahlfors regular metric spaces.  The main result of the paper demonstrates that if a singular integral operator on a Ahlfors regular metric space satisfies an additional estimate, then knowing the ``T(1)'' conditions for the operator imply that the operator is bounded on $L^2$.  The method of proof of the main result is an extension and another application of the work originated by Nazarov, Treil and the first author on non-homogeneous harmonic analysis.

\end{abstract}

\maketitle

\section{Introduction and Statements of results}
\label{1}

We are interested in \cz\, operators living on metric spaces.  In particular, these kernels will live on a metric space of homogeneous type.  We briefly recall these types of metric spaces.  A metric space of homogeneous type is a space $X$, a quasi-metric $\rho$, and a non-negative Borel measure $\nu$ on the space $X$.  The key property that defines these spaces is that all balls $B(x,r)$ defined by $\rho$ are open, and the measure $\nu$ satisfies the doubling condition
$$
\nu(B(x,2r))\leq C_{doub}\nu(B(x,r))\quad\forall x\in X,\quad r\in\mathbb{R}_+.
$$ 
We also require that $\nu(B(x,r))<\infty$ for all $x\in X$ and $r\in\mathbb{R}_+$.  The main example of the metric spaces that the reader should keep in mind is the case of $\mathbb{R}^n$ with the standard metric and Lebesgue measure.  Instead of the standard doubling condition, we will impose a slightly stronger condition.

Let $(X,\nu,\rho)$ be a Ahlfors regular metric measure space. By this we mean that $(X,\rho)$ is a complete metric space, $\nu\ge0$ is a Borel measure on $X$, and there exist constants $0<c_1<c_2$, $n>0$, such that, for all $r\ge0$ and $x\in X$:
\begin{equation}
 \label{ahlfors}
c_1r^n\le \nu(B(x,r))\le c_2 r^n.
\end{equation}
It is easy to see that condition \ref{ahlfors} implies the doubling condition on $\nu$ with $C_{doub}=\frac{c_2}{c_1}2^n$.

We next recall the definition of \cz\, operators on metric spaces as introduced by Christ, \cite{C}.  For any $x,y\in X$, we set
$$
\lambda(x,y)=\nu(B(x,\rho(x,y)))\approx \rho(x,y)^{n}.
$$
A simple calculation shows that $\lambda(x,y)\approx\lambda(y,x)$ because of the doubling condition on $\nu$.  Then a \textit{standard kernel} is a function $k:X\times X\setminus\{x=y\}\to\mathbb{C}$ such that there exists constants $C_{CZ}$, $\tau, \delta>0$
$$
|k(x,y)|\le\frac{C_{CZ}}{\lambda(x,y)}=\frac{C_{CZ}}{\rho(x,y)^n}\quad\forall x\neq y\in X;
$$
and
$$
|k(x,y)-k(x,y')|+|k(x,y)-k(x',y)|\le C_{CZ}\frac{\rho(x,x')^\tau}{\rho(x,y)^{\tau}}\frac{1}{\lambda(x,y)}=C_{CZ}\frac{\rho(x,x')^\tau}{\rho(x,y)^{\tau+n}}
$$
provided that $\rho(x,x')\leq \delta\rho(x,y)$.  In this situation, we say that the kernel $k$ satisfies the standard estimates.  Again, the canonical examples to keep in mind are the usual \cz\, kernels on $\mathbb{R}^n$.

However, we will be interested in kernels that satisfy estimates as if they lived on a ``smaller space''.  
First, suppose that we have another measure $\mu$ on the metric space $X$ (which need not be doubling), but satisfies the following relationship, for some $0\leq m<n$
\begin{equation*}
\tag{H}
\label{nonhom}
\mu\left(B(x,r)\right)\lesssim r^m\quad \forall x\in X,\quad\forall r.
\end{equation*}


Then, we define a \textit{standard kernel of order} $0<m\leq n$ as a function $k:X\times X\setminus\{x=y\}\to\mathbb{C}$ such that there exists constants $C_{CZ}$, $\tau, \delta>0$
$$
|k(x,y)|\le\frac{C_{CZ}}{\rho(x,y)^m}\quad\forall x\neq y\in X;
$$
and
$$
|k(x,y)-k(x,y')|+|k(x,y)-k(x',y)|\le C_{CZ}\frac{\rho(x,x')^\tau}{\rho(x,y)^{\tau+m}}
$$
provided that $\rho(x,x')\leq \delta\rho(x,y)$.  In this situation, we say that the kernel $k$ satisfies the standard estimates.  In this case, we then define the \cz\; operator associated to $\mu$ as
$$
T_\mu(f)(x):=\int_{X} k(x,y)f(y)d\mu(y).
$$
For ``nice'' functions $f$, this integral is well defined and 

These definitions are motivated by the \cz\, kernels that live in $\mathbb{R}^n$, but satisfy estimates as if they lived in $\mathbb{R}^m$ with $m\leq n$.  One should think of the measure $\mu$ as given by the $m$-dimensional Lebesgue measure after restricting to a $m$-dimensional hyperplane.

The constants $C_{CZ}$, $\tau$, $\delta$ and $m$ will be referred to as the \cz\, constants of the kernel $k(x,y)$.

We will also be interested in the kernels that have the additional property that satisfy 
$$
|k(x,y)|\le \frac1{\max (d^m(x), d^m(y))}\,,
$$
where $d(x):=\dist (x, X\setminus \Omega)=\inf\{\rho(x,y):y\in X\setminus \Omega\}$ and $\Omega$ being an open set in $X$.  

Our main result is the following theorem:


\begin{theorem}
\label{md}
Let $(X,\rho,\nu)$ be a Ahlfors regular metric space.  Let $k(x,y)$ be a \cz\, kernel of order $m$ on $(X,\rho,\nu)$, with \cz\, constants $C_{CZ}$ and $\tau$, that satisfies
$$
|k(x,y)|\le \frac1{\max (d(x)^m, d(y)^m)}\,,
$$
where $d(x):=\dist (x, X\setminus \Omega)$.
Let $\mu$ be a probability measure with compact support in $X$ and all balls such that $\mu(B(x,r)) >r^m$ lie in an open set $\Omega$.  Finally, suppose also that a ``$T1$ Condition'' holds for the operator $T_{\mu,m}$ with kernel $k$ and for the operator $T^*_{\mu,m}$ with kernel $k(y,x)$:
\begin{equation}
\label{T1}
\|T_{\mu,m}\chi_Q\|_{L^2(X;\mu)}^2 \le A\,\mu(Q)\,,\, \|T^*_{\mu,m}\chi_Q\|_{L^2(X;\mu)}^2 \le A\,\mu(Q)\,.
\end{equation}
Then $\|T_{\mu,m}\|_{L^2(X;\mu)\rightarrow L^2(X;\mu)} \le C(A,m,d,\tau)$.
\end{theorem}

The balls for which we have $\mu(B(x,r))>r^m$ will be called ``non-Ahlfors balls".  The key hypothesis is that we can capture all the non-Ahlfors balls in some open set $\Omega$.  To mitigate against this difficulty, we will have to suppose that our \cz\, kernels have an additional estimate in terms of the behavior in terms of  the distance to the complement of $\Omega$.

An immediate application of Theorem \ref{md} is a new proof of results by the authors in \cite{VW}.  In \cite{VW} a variant of Theorem \ref{md} was obtained in the Euclidean setting, and then is further extended to \cz\, kernels in the natural metric associated to the Heisenberg group on the unit ball.  This was then used to characterize the Carleson measures for the analytic Besov--Sobolev spaces on the unit ball in $\mathbb{C}^n$.  The connection between Carleson measures and a variant of Theorem \ref{md} is provided since a measure is Carleson if and only if a certain naturally occurring \cz\, operator is bounded on $L^2$.  The operator to be studied is amenable to the methods of non-homogeneous harmonic analysis.

The method of proof of Theorem \ref{md} will be to use the tools of non-homogeneous harmonic analysis as developed by F. Nazarov, S. Treil, and the first author in the series of papers \cites{NTV1, NTV2, NTV3, NTV4} and further explained in the book by the first author \cite{V}.  We essentially adapt the proof given by the authors in \cite{VW} to the case of metric spaces considered in this paper.  Constants will be denoted by $C$ throughout the paper.

\section{Proof of Theorem \ref{md}}
The proof of this theorem will be divided in several parts.  We first recall the construction of M. Christ of ``dyadic cubes'' on a metric space of homogeneous type, see \cite{C}.  The interested reader can also consult the paper by E. Sawyer and R. Wheeden, \cite{SW}, where a similar construction is performed.

\begin{theorem}[M. Christ, \cite{C}]
\label{dyadicmetric}
There exists a collection of open sets $\{Q_\alpha^k\subset X: k\in\mathbb{Z}, \alpha\in I_k\}$ and constants $\kappa\in (0,1)$, $a_0>0$, and $\eta>0$ and $C_1, C_2<\infty$ such that
\begin{itemize}
\item[(i)] $\nu (X\setminus\bigcup_\alpha Q_\alpha^k)=0\quad\forall k\in\mathbb{Z}$;
\item[(ii)] If $l\geq k$, then either $Q_\beta^l\subset Q_\alpha^k$ or $Q_\beta^l\subset Q_\alpha^k=\emptyset$;
\item[(iii)] For each $(k,\alpha)$ and each $l<k$ there is a unique $\beta$ such that $Q_\alpha^k\subset Q_\beta^l$;
\item[(iv)] The diameter of $Q_\alpha^k$ is an absolute constant multiple of $\kappa^k$;
\item[(v)] Each $Q_\alpha^k$ contains some ball $B(z_\alpha, a_0\kappa^k)$;
\item[(vi)] $\nu\{x\in Q_\alpha^k:\dist(x, X\setminus Q_\alpha^k)\leq t\kappa^k\}\lesssim t^\eta\nu(Q_\alpha^k)$
\end{itemize}
Here $I_k$ is a (possibly finite) index set, depending only on $k\in\mathbb{Z}$.
\end{theorem}
The construction of these cubes uses only the properties of the homogeneous space $(X,\rho,\nu)$.  One can think of the cubes $Q_\alpha^k$ as being cubes or balls of diameter $\kappa^k$ and center $z_\alpha^k$.  We will let $\mathcal{D}$ denote the collection of dyadic cubes on $X$ that exists by the above Theorem.

We further remark that it is possible to ``randomize'' this construction.  In a recent paper by Hyt\"onen and Martikainen, \cite{HM}, they studied this construction in and showed that it is possible to construct several random dyadic grids of the type above.  The details of this construction aren't immediately important for the proof of the main results in this paper, only the existence of these random grids.  We recommend that the reader consult the well-written paper \cite{HM} for the construction of these grids.  In particular, Section 10 of that paper contains the necessary modifications of Theorem \ref{dyadicmetric} to construct the random dyadic lattices in a metric space.

We also define the dilation of a set $E\subset X$ by a parameter $\lambda\geq 1$ by
$$
\lambda E:=\{x\in X:\rho(x,E)\leq(\lambda-1)\textnormal{diam}(E)\}.
$$

\subsection{Terminal and transit cubes}
\label{tt}

We will call the cube $Q\in \Dk$ a {\it terminal cube} if the parent of $Q$ (which exists and is unique by (iii) of Theorem \ref{dyadicmetric}) is contained in our open set $\Omega$ or $\mu(Q)=0$.  All other cubes are called {\it transit} cubes.   Then, denote by $\Dk^{term}$ and $\Dk^{tran}$ as the terminal and transit cubes from $\Dk$.   We first state two obvious Lemmas.

\begin{lemma}
\label{obv1}
If $Q$ belongs to $\Dk^{term}$, then
$$
|k(x,y)|\le \frac1{\kappa^m}\,.
$$
\end{lemma}
This follows since $Q$ belongs to its parent which is a subset of $\Omega$ and so for $x,y\in Q$ we have that $d(x)\geq\kappa$ and similarly for $y$.  Another obvious lemma:

\begin{lemma}
\label{obv2}
If $Q$ belongs to $\Dk^{tran}$, then
$$
\mu(B(x,r)) \lesssim r^m\,.
$$
\end{lemma}

\medskip

We assume that $F=\supp\mu$ lies in a grand child cube of $Q$ where, this $Q$ is a certain (fixed) unit cube. We then take two ``random'' lattices as constructed by Hyt\"onen and Martikainen in \cite{HM}.  Now, let $\Dk_1$ and $\Dk_2$ be two such dyadic lattices, that have the property that the unit cube contains the support of $\mu$ deep inside a unit cube of the corresponding lattice.  We will decompose our functions $f$ and $g$ with respect to the lattices $\Dk_1$ and $\Dk_2$.

We would like to denote $Q_j$ as a dyadic cube belonging to the dyadic lattice $\Dk_j$.  Unfortunately, this makes the notation later very cumbersome.  So, we will use the letter $Q$ to denote a dyadic cube belonging to the lattice $\Dk_1$ and the letter $R$ to denote a dyadic cube belonging to the lattice $\Dk_2$.  We will also let $s(Q)$ denote the ``size'' or ``scale'' of the cube, namely, what generation of the construction from Theorem \ref{dyadicmetric} the cube belongs to.

From now on, we will always denote by $Q_j$ the dyadic subcubes of a cube $Q$ enumerated in some ``natural order''.   Similarly, we will always denote by $R_j$ the dyadic subcubes of a cube $R$ from $\mathcal{D}_2$.  

Next, notice that there are special unit cubes $Q^0$ and $R^0$ of the dyadic lattices $\Dk_1$ and $\Dk_2$ respectively.   They have the property that they are both transit cubes and contain $F$ deep inside them.

\subsection{Projections $\La$ and $\Delta_Q$}
\label{pr}

Let $\mathcal{D}$ be one of the dyadic lattices above. For a function $\psi\in L^1(X;\mu)$ and for a cube $Q\subset X$, denote by
$\langle \psi\rangle\ci Q$ the average value of $\psi$ over $Q$ with respect to the measure $\mu$, i.e.,

$$
\langle \psi\rangle\ci Q := \frac{1}{\mu(Q)}\int_Q\psi\,d\mu
$$
(of course, $\langle \psi\rangle\ci Q$ makes sense only for cubes $Q$ with $\mu(Q)>0$).
Put
$$
\Lambda\vf :=\langle \vf\rangle\ci {Q^0}\,.
$$
Clearly, $\Lambda\vf\in L^2(X;\mu)$ for all $\vf\in L^2(X;\mu)$, and $\Lambda^2=\Lambda$, i.e., $\Lambda$ is a projection. Note also, that actually $\Lambda$
does not depend on the lattice $\mathcal{D}$ because the average is taken over the whole support of the measure $\mu$ regardless of the position of the cube $Q^0$ (or $R^0$).

Below we will start almost every claim by ``Assume (for definiteness) that $s(Q)\le s(R)$\dots ''.  Below, for ease of notation, we will write that a cube $Q\in\mathcal{X}\cap\mathcal{Y}$ to mean that the dyadic cube $Q$ has both property $\mathcal{X}$ and $\mathcal{Y}$ simultaneously.  

For every transit cube $Q\in\Dk_1$, define $\Delta\ci Q\vf$ by
$$
\Delta\ci Q\vf\bigr|\ci{X\setminus Q}:=0,\qquad\,,\,\,\,
\Delta\ci Q\vf\bigr|\ci{Q_j}:=\left\{\aligned\left[\langle \vf\rangle\ci {Q_j} -\langle \vf\rangle\ci {Q}\right] &\text{\quad if $Q_j$ is transit;}\\ \vf-\langle \vf\rangle\ci {Q} &\text{\quad if $Q_j$ is terminal.}\endaligned\right.$$

Observe that for every transit cube $Q$, we have $\mu(Q)>0$, so our definition makes sense since no zero can appear in the denominator.  We repeat the same definition for $R\in \mathcal{D}_2$.

We then have have following Lemma that collects several easy properties of $\Delta\ci Q\vf$.  To check these properties is left to the reader as an exercise.

\begin{lemma}
\label{easyproperties}
For every $\vf\in L^2(X;\mu)$ and every transit cube $Q$,
\begin{itemize}
\item[(1)] \emph{$\Delta\ci Q\vf\in L^2(X;\mu)$};
\item[(2)] \emph{$\int_{X}\Delta\ci Q\vf\,d\mu=0$};
\item[(3)] \emph{$\Delta\ci Q$ is a projection, i.e., $\Delta\ci Q^2=\Delta\ci Q$};
\item[(4)] \emph{$\Delta\ci Q\Lambda=\Lambda\Delta\ci Q=0$};
\item[(5)]  \emph{If $Q,\widetilde{Q}$ are transit, $\widetilde{Q}\neq Q$, then $\Delta\ci Q\Delta\ci {\widetilde{Q}}=0$}.
\end{itemize}
\end{lemma}

We next note that it is possible to decompose functions $\varphi$ into the corresponding projections $\Lambda$ and $\Delta\ci Q$.

\begin{lemma}
\label{Riesz}
Let $Q^0$ be a transit cube. For every $\vf\in L^2(X;\mu)$ we have
$$
\vf=\Lambda\vf+\sum_{Q \,\text{transit}}\Delta\ci Q\vf,
$$
the series converges in $L^2(X;\mu)$ and, moreover,
$$
\|\vf\|^2\ci{L^2(\mu)}=
\|\Lambda\vf\|^2\ci{L^2(\mu)}+\sum_{Q\,\text{transit}}\|\Delta\ci Q\vf\|^2\ci{L^2(\mu)}\,.
$$
\end{lemma}

\begin{proof}
Note first of all that if one understands the sum 
$$\sum_{Q\,\text{transit}}$$
as $\lim_{k\to\infty}\sum_{Q\,\text{transit}:s(Q)> \delta^{k}}$, then for
$\mu$-almost every $x\in X$, one has
$$
\vf(x)=\Lambda\vf(x)+\sum_{Q\,\text{transit}}\Delta\ci Q\vf( x).
$$
Indeed, the claim is obvious if the point $x$ lies in some terminal cube.
Suppose now that this is not the case. Observe that

$$
\Lambda\vf(x)+\sum_{Q\,\text{transit}:s(Q)> \kappa^{k}}\Delta\ci Q\vf(x)=
\langle \vf\rangle\ci{Q^{k}} ,
$$
where $Q^{k}$ is the dyadic cube of size $\kappa^{k}$, containing $x$.
Therefore, the claim is true if
$$
\langle \vf\rangle\ci{Q^{k}}\to \vf(x)\,.
$$
But, the exceptional set for this condition has $\mu$-measure $0$.  Now the orthogonality of all $\Delta_Q\vf$ between themselves, and their orthogonality to $\La\vf$ proves the lemma.
\end{proof}

\section{Good and bad functions}
\label{gb}

We consider the functions $f$ and $g \in L^2(X;\mu)$. We fix two dyadic lattices $\Dk_1$ and $\Dk_2$ as before and define decompositions of $f$ and $g$ via Lemma \ref{Riesz},
$$
f= \La f +\sum_{Q\in \Dk_1^{tran}} \Delta_Q f,\quad g=\La g +\sum_{R\in \Dk_2^{tran}} \Delta_R g.
$$

For a  dyadic cube $R$ we denote $\cup_{i\in I_k} \pd R_i$ by $sk \,R$, called the \textit{skeleton} of $R$. Here the $R_i$ are the dyadic children of $R$. 
 
Let $\tau, m$ be parameters of the \cz \,kernel $k$. We fix $\alpha= \frac{\tau}{2\tau+2m}$.

\begin{defi}
Fix a small number $\delta>0$ and $S\ge 2$ to be chosen later.  Choose an integer $r$ such that 
\begin{equation}
\label{r}
\kappa^{-r}\le \delta^S < \kappa^{-r+1}\,.
\end{equation}
A cube $Q\in \Dk_1$ is called {\it bad} ($\delta$-bad) if there exists $R\in \Dk_2$ such that
\begin{itemize}
\item[(1)] $s(R)\ge \kappa^r s(Q)$;
\item[(2)] $\dist (Q, sk\,R) <s(Q)^{\alpha} s(R)^{1-\alpha}\,$.
\end{itemize}
\end{defi}
Let $\mathcal{B}_1$ denote the collection of all bad cubes and correspondingly let $\mathcal{G}_1$ denote the collection of good cubes.  The symmetric definition gives the collection of {\it bad} cubes $R\in \Dk_2$, denotes as $\mathcal{B}_2$.

\bigskip

We say, that $\vf = \sum_{Q\in \Dk_1^{tran}} \Delta_Q\vf$ is \textit{bad} if in the sum only bad $Q$'s participate in this decomposition with the same appling to $\psi = \sum_{Q\in \Dk_2^{tran}} \Delta_Q\psi$. In particular, given two distinct lattices $\Dk_1$ and $\Dk_2$ we fix the decomposition of $f$ and $g$ into good and bad parts:
$$
f= f_{good} + f_{bad}\,,\,\,\text{where}\,\, f_{good} = \La f + \sum_{Q\in \Dk_1^{tran}\cap\mathcal{G}_1} \Delta_Q f\,.
$$
The same applies to $g = \La g + \sum_{R\in \Dk_2^{tran}} \Delta_R g= g_{good} +g_{bad}$.

\begin{theorem}
\label{badprob}
One can choose $S= S(\alpha)$ in such a way that for any fixed $Q\in \mathcal{D}_1$,
\begin{equation}\label{prob}\mathbb{P}\{Q \,\text{is bad}\} \leq \delta^2\,.
\end{equation}
By symmetry $\mathbb{P}\{R \,\text{is bad}\} \leq \delta^2$ for any fixed $R\in \mathcal{D}_2$.
\end{theorem}

The proof of this Theorem can be found in the paper \cite{HM}.    The use of Theorem \ref{badprob} gives us
 $S= S(\alpha)$ in such a way that for any fixed $Q\in \mathcal{D}_1$,
 \begin{equation}
 \label{probagain}
 \mathbb{P}\{Q \,\text{is bad}\} \leq \delta^2\,.
 \end{equation}
We are now ready to prove

\begin{theorem}
\label{badprobfagain}
Consider the decomposition of $f$ from Lemma \ref{Riesz}.  Then one can choose $S= S(\alpha)$ in such a way that 
\begin{equation}
\label{probf}
\mathbb{E}(\|f_{bad}\|_{L^2(X;\mu)}) \leq \delta\|f\|_{L^2(X;\mu)}\,.
\end{equation}
\end{theorem}
The proof depends only on the property \eqref{probagain} and not on a particular definition of what it means to be a bad or good function.
\begin{proof}
By Lemma \ref{Riesz} (its left inequality),
$$
\mathbb{E}(\|f_{bad}\|_{L^2(X;\mu)}) \leq \mathbb{E}\Big(\sum_{Q\in\Dk_1^{tran}\cap\mathcal{B}_1}\|\Delta_Q f\|^2_{L^2(X;\mu)}\Big)^{1/2}\,.
$$
Then
$$
\mathbb{E}(\|f_{bad}\|_{L^2(X;\mu)}) \leq \Big(\mathbb{E}\sum_{Q\in\Dk_1^{tran}\cap\mathcal{B}_1}\|\Delta_Q f\|^2_{L^2(X;\mu)}\Big)^{1/2}\,.
$$
Let $Q$ be a fixed cube in $\mathcal{D}_1$; then, using \eqref{probagain}, we conclude: 
$$
\mathbb{E}\|\Delta_Q f\|^2_{L^2(X;\mu)}=\mathbb{P}\{Q\,\text{is bad}\} \|\Delta_Q f\|^2_{L^2(X;\mu)} \leq \delta^2 \|\Delta_Q f\|^2_{L^2(X;\mu)}\,.
$$
 Therefore, we can continue as follows:
 $$
 \mathbb{E}(\|f_{bad}\|_{L^2(X;\mu)}) \leq \delta\Big(\sum_{Q\in\Dk_1^{tran}\cap\mathcal{B}_1}\|\Delta_Q f\|^2_{L^2(X;\mu)}\Big)^{1/2}\leq \delta\|f\|_{L^2(X;\mu)}\,.
 $$
 The last inequality uses Lemma \ref{Riesz} again (its right inequality).
\end{proof}
This theorem can also be found in the paper \cite{HM}.

\subsection{Reduction to Estimates on Good Functions}
\label{redtogood}

We consider two random dyadic lattices $\Dk_1$ and  $\Dk_2$ as constructed in \cite{HM}.  Take now two functions $f$ and $g\in L^2(X;\mu)$ decomposed according to Lemma \ref{Riesz} 
$$
f=\La f + \sum_{Q\in \Dk_1^{tran}} \Delta_Q f\,,\,\,g=\La g + \sum_{R\in \Dk_2^{tran}} \Delta_R g\,.
$$
Recall that we can now write $f=f_{good} + f_{bad}$, $g=g_{good}+ g_{bad}$.  Then 
$$
(Tf,g) = (Tf_{good}, g_{good}) + R(f,g)\,,\,\,\text{where}\,\, R(f,g)= (Tf_{bad}, g) + (Tf_{good}, g_{bad})\,.
$$

\begin{theorem}
\label{R}
Let $T$ be any operator with bounded kernel. Then
$$
\mathbb{E} |R(f, g)| \le 2\,\delta\|T\|_{L^2(X;\mu)\to L^2(X;\mu)} \|f\|_{L^2(X;\mu)} \|g\|_{L^2(X;\mu)}\,.
$$ 
\end{theorem}

\begin{remark}
Notice that the estimate depends on the norm of $T$ not on the bound on its kernel.
\end{remark}

\begin{proof}
The procedure of taking the good and bad part of a function are projections in $L^2(X;\mu)$ and so they do not increase the norm.  Since we have that the operator $T$ is bounded, then
$$
|R(f,g)|\leq \|T\|_{L^2(X;\mu)\to L^2(X;\mu)}\left(\|g\|_{L^2(X;\mu)}\|f_{bad}\|_{L^2(X;\mu)} + \|f\|_{L^2(X;\mu)}\|g_{bad}\|_{L^2(X;\mu)}\right)
$$ 
Therefore, upon taking expectations we find
$$
\mathbb{E}|R(f,g)| \le \|T\|_{L^2(X;\mu)\to L^2(X;\mu)}\left(\|g\|_{L^2(X;\mu)}\mathbb{E}(\|f_{bad}\|_{L^2(X;\mu)}) + \|f\|_{L^2(X;\mu)}\mathbb{E}(\|g_{bad}\|_{L^2(X;\mu)})\right)\,.
$$
Using Theorem \ref{badprobfagain} we finish the proof.

\end{proof}

We see that we need now only to estimate
\begin{equation}
\label{good}
|(Tf_{good}, g_{good})| \le C(\tau,m,d, T1) \|f\|_{L^2(X;\mu)}\|g\|_{L^2(X;\mu)}\,.
\end{equation}
In fact, considering any operator $T$ with bounded kernel we conclude
$$
(Tf,g) = \mathbb{E}(Tf,g) = \mathbb{E}(Tf_{good}, g_{good}) + \mathbb{E} R(f,g)\,.
$$
Using Theorem \ref{R} and \eqref{good} we have
$$
|(Tf,g)| \le C\,\|f\|_{L^2(X;\mu)}\|g\|_{L^2(X;\mu)} + 2\delta\|T\|_{L^2(X;\mu)\to L^2(X;\mu)}\|f\|_{L^2(X;\mu)}\|g\|_{L^2(X;\mu)}\,.
$$
From here, taking the supremum over $f$ and $g$ in the unit ball of $L^2(X;\mu)$, and choosing $\delta=\frac14$ we get
$$
\|T\|_{L^2(X;\mu)\to L^2(X;\mu)}\le 2C\,.
$$

\subsection{Splitting $(Tf_{good}, g_{good})$ into Three Sums}
\label{spli}

First let us get rid of the projection $\La$. We fix two corresponding dyadic lattices $\Dk_1$ and $\Dk_2$.  Recall that $F=\supp\mu$ is deep inside a unit cube $Q$ of the standard dyadic lattice $\Dk$ as well as inside the shifted unit cubes $Q^0\in \Dk_1$ and $R^0\in \Dk_2$.  If $f\in L^2(X;\mu)$, we have
\begin{eqnarray*}
\|T\La f\|_{L^2(X;\mu)}  & = & \langle f\rangle_{Q^0} \|T\chi_{Q^0}\|_{L^2(X;\mu)}\\
 & \leq &  A_{\ref{T1}}\frac{\|f\|_{L^2(X;\mu)} \mu(Q^0)^{1/2}}{\mu(Q^0)} \mu(Q^0)^{1/2}\\
 & = & A_{\ref{T1}}\|f\|_{L^2(X;\mu)}\,.
\end{eqnarray*} 

So we can replace $f$ by $f-\La f$ and identically we can repeat this argument with $g$ and from now on we may assume further that
$$
\int_X f(x)\,d\mu(x)=0\textnormal{ and }\int_X g(x)\,d\mu(x)=0\,.
$$
Based on the reductions above, we can now think that $f$ and $g$ are good functions with zero averages. We skip mentioning below that $Q\in \Dk_1^{tran}$ and $R\in \Dk_2^{tran}$, since this will always be the case by the convention established above.

To study the action of the \cz\, operator $T$ on $f$ and $g$, we split the pairing in the following manner,
\begin{align*}
(Tf,g)= \sum_{Q \in\mathcal{G}_1, R\in \mathcal{G}_2 , s(Q) \le s(R)} (\Delta_Q f, \Delta_R g) +
\sum_{Q \in\mathcal{G}_1, R\in \mathcal{G}_2, s(Q) > s(R)} (\Delta_Q f, \Delta_R g) \,.
\end{align*}
The question of convergence of the infinite sum can be avoided here, as we can think that the functions $f$ and $g$ are only finite sums.  This removes the question of convergence and allows us to rearrange and group the terms in the sum in any way we want.

We need to estimate only the first sum, as the second will follow by symmetry. For the sake of notational simplicity we will skip mentioning that the cubes $Q$ and $R$ are good and we will skip mentioning $s(Q)\le s(R)$. So, for now on, 
$$
\sum_{Q,R: \text{other conditions}}\textnormal{ means } \sum_{Q,R: s(Q)\le s(R),\, Q \in\mathcal{G}_1, \,R\in\mathcal{G}_2 , \,\text{other conditions}}\,.
$$

\begin{remark}
It is convenient sometimes to think that the summation 
$$
\sum_{Q,R: \,\text{other conditions}}
$$
goes over good $Q$ and \textit{all} $R$. Formally, this does not matter, since the functions $f$ and $g$ are good functions, and so this merely reduces to adding or omitting several zeros to the sum. For the symmetric sum over $Q,R: s(Q)>s(R)$ the roles of $Q$ and $R$ in this remark must of course be interchanged.
\end{remark}

The definition of $\delta$-badness involved a large integer $r$, see \eqref{r}. Use this notation to write our sum over $s(Q)\le s(R)$ as follows
\begin{align*}
\sum_{Q,R} (\Delta_Qf, \Delta_R g) = \sum_{Q,R: s(Q)\ge \kappa^{-r} s(R)} + \sum_{Q,R: s(Q)< \kappa^{-r} s(R)}=
\sum_{Q,R: s(Q)\ge \kappa^{-r} s(R),\,\dist(Q,R) \le s(R)}+\\
\bigg[\sum_{Q,R: s(Q)\ge \kappa^{-r} s(R),\,\dist(Q,R)>s(R)}+\sum_{Q,R: s(Q)< \kappa^{-r} s(R),\,Q\cap R=\emptyset}\bigg] + \sum_{Q,R: s(Q)< \kappa^{-r} s(R),\,Q\cap R\neq\emptyset}\\
=:\sigma_1 +\sigma_2 +\sigma_3\,.
\end{align*}

\subsection{Three Potential Estimates of $\int_X\int_X k(x,y) f(x)g(y)\,d\mu(x)\,d\mu(y)$}
\label{threetypes}
Recall that the kernel $k(x,y)$ of $T$ satisfies the estimate
$$
|k(x,y)|\le \frac{1}{\max (d(x)^m, d(y)^m)}\,,\,\,\,d(x) =\dist (x,X\setminus \Omega)\,,
$$ $\Omega$ being an open set in $X$, and 
$$
|k(x,y)|\le\frac{C_{CZ}}{\rho(x,y)^m}\quad\forall x\neq y\in X;
$$
and
$$
|k(x,y)-k(x,y')|+|k(x,y)-k(x',y)|\le C_{CZ}\frac{\rho(x,x')^{\tau}}{\rho(x,y)^{\tau+m}}
$$
provided that $\rho(x,x')\leq \delta\rho(x,y)$, with some fixed constants numbers $C_{CZ}, \tau, m$.

First, we will sometimes write 
$$
\int_X\int_X k(x,y) f(x)g(y)\,d\mu(x)\,d\mu(y)=\int_X\int_X [k(x,y)-k(x_0,y)]f(x)g(y)\,d\mu(x)\,d\mu(y)
$$
using the fact that our $f$ and $g$ will actually be $\Delta_Q f$ and $\Delta_R g$ and so their
integrals are zero. Temporarily write $K(x,y)$ for either $k(x,y)$ or $k(x,y)-k(x_0,y)$.

After that we have three logical possibilities to estimate 
$$
\int_X\int_X K(x,y) f(x)g(y)\,d\mu(x)\,d\mu(y)\,.
$$
\begin{itemize}
\item[(1)] Estimate $|K|$ in $L^{\infty}$, and $f, g$ in $L^1$ norms;
\item[(2)] Estimate $|K|$ in $L^{\infty}L^1$ norm, and $f$ in $L^1$ norm, $g$ in $L^{\infty}$ norm (or maybe, do this symmetrically); 
\item[(3)] Estimate $|K|$ in $L^1$ norm, and $f, g$ in $L^{\infty}$ norms.
\end{itemize}

The third method is widely used for Calder\'on--Zygmund estimates on homogeneous spaces (say with respect to Lebesgue measure), but it is very dangerous to use in the case of a nonhomogeneous measure. Here is the reason. After $f$ and $g$ are estimated in the $L^{\infty}$ norm, one needs to continue these estimates to have $L^2$ norms. There is nothing strange in that as usually $f$ and $g$ are almost proportional to characteristic functions. But for $f$ living on $Q$ such that $f=c_Q\chi_Q$ ($c_Q$ is a constant),
$$
\|f\|_{L^{\infty}(X:\mu)} \leq \frac1{\mu(Q)^{1/2}}\|f\|_{L^2(X;\mu)}\,.
$$
The same reasoning applies for $g$ on $R$. Then
$$
\Big|\int_X\int_X K(x,y) f(x)g(y)\,d\mu(x)\,d\mu(y)\Big|\leq \frac1{\left(\mu(Q)\mu(R)\right)^{1/2}}\|f\|_{L^2(X;\mu)}\|g\|_{L^2(X;\mu)}\,.
$$
And the nonhomogeneous measure has no estimate from below. Having two uncontrollable almost zeroes in the denominator is a very bad idea.  We will never use the estimate of type (3). 

On the other hand, estimates of type (2) are much less dangerous (although requires the care as well).  This is because, in this case one applies 
 $$
 \|f\|_{L^{1}(X;\mu)} \leq \mu(Q)^{1/2}\|f\|_{L^2(X;\mu)}\textnormal{ and } \|g\|_{L^{\infty}(X;\mu)} \leq\frac1{\mu(R)^{1/2}}\|g\|_{L^2(\mu)}\,,$$and gets 
 $$
 \Big|\int_X\int_X K(x,y) f(x)g(y)\,d\mu(x)\,d\mu(y)\Big|\leq \left(\frac{\mu(Q)}{\mu(R)}\right)^{1/2}\|f\|_{L^2(X;\mu)}\|g\|_{L^2(X;\mu)}\,.
 $$
 If we choose to use estimate of the type (2) only for pairs $Q,R$ such that $Q\subset R$ we are in good shape. This approach is what we will end up going when estimating $\sigma_3$.
 
\bigskip

\noindent{\bf Plan}. The first sum is the ``diagonal" part of the operator, $\sigma_1$. The second sum, $\sigma_2$ is the ``long range interaction".  The final sum, $\sigma_3$, is the ``short range interaction".  The diagonal part will be estimated using our $T1$ assumption of Theorem \ref{T1}, for the long range interaction we will use the first type of estimates described above, for the short range interaction we will use estimates of types (1) and (2) above. But all this will be done carefully!

\section{The Long Range Interaction: Controlling Term $\sigma_2$}
\label{lr}

We first prove a lemma that demonstrates that for functions with supports that are far apart, we have some good control on the bilinear form induced by our \cz\, operator $T$.  For two dyadic cubes $Q$ and $R$, we set
$$
D(Q,R):=s(Q)+s(R)+\dist(Q,R).
$$ 

\begin{lemma} 
\label{FarInteractionLemma}
Suppose that $Q$ and $R$ are two cubes in $X$, such that $s(Q) \leq s(R)$. Let $\vf\ci Q,\psi\ci R\in L^2(X;\mu)$. Assume that $\vf\ci Q$ vanishes outside $Q$, and $\psi\ci R$ vanishes outside $R$; $\int_{X}\vf\ci Qd\mu=0$ and, at last, $\dist(Q,\supp\psi\ci R)\ge s(Q)^{\al}s(R)^{1-\al}$. Then
$$
|(\vf\ci Q,T\psi\ci R)|\le A\,C\,\frac{s(Q)^{\frac{\tau}{2}}s(R)^{\frac{\tau}{2}}}{D(Q,R)^{m+\tau}}\sqrt{\mu(Q)\mu(R)}\|\vf\ci Q\|\ci{L^2(X;\mu)}\|\psi\ci R\|\ci{L^2(X;\mu)}.
$$
\end{lemma} 
  
\begin{remark}
Note that we require only that the support of the function $\psi_R$ lies far from the cube $Q$; the cubes $Q$ and $R$ themselves may intersect!  Such situations will arise when estimating the term $\sigma_2$.
\end{remark}

\begin{proof}

Let $x\ci Q$ be the center of the cube $Q$. Note that for all $x\in Q$,
$y\in \supp\psi\ci R$, we have
$$
\rho(x\ci Q,y)\ge \frac{s(Q)}{2}+\dist(Q,\supp\psi\ci R)\ge
\frac{s(Q)}{2} + 2^{r(1-\al)}s(Q)\gtrsim s(Q)\gtrsim\rho(x,x\ci Q).
$$
Therefore,
\begin{align*}
|( \vf\ci Q, T \psi\ci R )|&=
\Bigl|\int_X\int_X k(x,y)\vf\ci Q(x)\psi\ci R(y)\, d\mu(x)\, d\mu(y)\Bigr|\\
&=
\Bigl|\int_X\int_X [k(x,y)-k(x\ci Q,y)]
\vf\ci Q(x)\psi\ci R(y)\, d\mu(x)\, d\mu(y)\Bigr|\\
&\lesssim
\frac{s(Q)^\tau}{\dist(Q,\supp\psi\ci R)^{m+\tau}}
\|\vf\ci Q\|\ci{L^1(X;\mu)}\|\psi\ci R\|\ci{L^1(X;\mu)}.
\end{align*}
There are two possible cases.

{\bf Case 1: $\dist(Q,\supp\psi\ci R)\ge s(R)$.}
Then
$$
D(Q,R):=s(Q)+s(R)+\dist(Q,R)\le
3\dist(Q,\supp\psi\ci R)
$$
and therefore
$$
\frac{s(Q)^\tau}{\dist(Q,\supp\psi\ci R)^{m+\tau}}\lesssim 
\frac{s(Q)^\tau}{D(Q,R)^{m+\tau}}\lesssim\frac{s(Q)^{\frac{\tau}{2}}s(R)^{\frac{\tau}{2}}}{D(Q,R)^{m+\tau}}.
$$

{\bf Case 2: $s(Q)^{\al}s(R)^{1-\al}\le \dist(Q,\supp\psi\ci R)\le s(R)$.}
Then $D(Q,R)\le 3s(R)$ and we get
$$
\frac{s(Q)^\tau}{\dist(Q,\supp\psi\ci R)^{m+\tau}}\le
\frac{s(Q)^\tau}{[s(Q)^{\al}s(R)^{1-\al}]^{m+\tau}}=
\frac{s(Q)^{\frac{\tau}{2}}s(R)^{\frac{\tau}{2}}}{s(R)^{m+\tau}}\lesssim
\frac{s(Q)^{\frac{\tau}{2}}s(R)^{\frac{\tau}{2}}}{D(Q,R)^{m+\tau}}.
$$
Here, key to the proof was the choice of $\alpha=\frac{\tau}{2(\tau+m)}$.  Now, to finish the proof of the lemma, it remains only to note that
$$
\|\vf\ci Q\|\ci{L^1(X;\mu)}\le
\sqrt{\mu(Q)}\|\vf\ci Q\|\ci{L^2(X;\mu)}
\text{ and }
\|\psi\ci R\|\ci{L^1(X;\mu)}\le
\sqrt{\mu(R)}\|\psi\ci R\|\ci{L^2(X;\mu)}.
$$

\end{proof}

Applying this lemma to $\vf\ci Q=\Delta\ci Q f$ and $\psi\ci R=\Delta\ci R g$, we
obtain
\begin{equation}
\label{twostar} 
|\sigma_2|\lesssim \sum_{Q,R}\frac{s(Q)^{\frac{\tau}{2}}s(R)^{\frac{\tau}{2}}}{D(Q,R)^{m+\tau}}\sqrt{\mu(Q)}\sqrt{\mu(R)}\|\Delta\ci Q f\|\ci{L^2(X;\mu)}\|\Delta\ci R g\|\ci{L^2(X;\mu)}\,.
\end{equation}

To control term $\sigma_2$ the computations above suggest that we will define a matrix operator, depending on the cubes $Q$ and $R$ and show that it is a bounded operator on $\ell^2$.

\begin{lemma}
\label{TQR}
Define
$$
T\ci{Q,R}:=
\frac{s(Q)^{\frac{\tau}{2}}s(R)^{\frac{\tau}{2}}}{D(Q,R)^{m+\tau}}
\sqrt{\mu(Q)}\sqrt{\mu(R)}\qquad (Q\in\mathcal{D}_1^{tr}  ,\,
R\in\mathcal{D}_2^{tr}  ,\,s(Q)\leq s(R)\,).
$$
Then, for any two families  $\{a\ci Q\}\ci{Q\in\mathcal{D}_1^{tr}  }$ and $\{b\ci R\}\ci{R\in\mathcal{D}_2^{tr}}$ of nonnegative numbers, one has
$$
\sum_{Q,R}T\ci{Q,R}a\ci Q b\ci R\le
A\,C\,\Bigl[\sum_{Q}a\ci Q^2\Bigr]^{\frac12}\Bigl[\sum_{R}b\ci R^2\Bigr]^{\frac12}.
$$
\end{lemma}

\begin{remark}
Note that $T\ci{Q,R}$ is defined for all $Q$ and $R$ with $s(Q)\le s(R)$
and that the conditions $\dist(Q,R)\ge s(Q)^{\al}s(R)^{1-\al}$ (or even the condition
$Q\cap R=\emptyset$) no longer appears as a condition in the summation!
\end{remark}

Assuming Lemma \ref{TQR} for the moment, the estimate of $\sigma_2$ then proceeds in an obvious fashion.  
\begin{eqnarray*}
|\sigma_2| & \lesssim & \sum_{Q,R}\frac{s(Q)^{\frac{\tau}{2}}s(R)^{\frac{\tau}{2}}}{D(Q,R)^{m+\tau}}\sqrt{\mu(Q)}\sqrt{\mu(R)}\|\Delta\ci Q f\|\ci{L^2(X;\mu)}\|\Delta\ci R g\|\ci{L^2(X;\mu)}\\
 & \lesssim & \left(\sum_{Q}\|\Delta\ci Q f\|\ci{L^2(X;\mu)}^2\right)^{1/2}\left(\sum_{R}\|\Delta\ci R g\|\ci{L^2(X;\mu)}^2\right)^{1/2}\\
 & \lesssim & \| f\|_{L^2(X;\mu)}\| g\|_{L^2(X;\mu)}.
\end{eqnarray*}
Here the first line follows by (\ref{twostar}), the second by Lemma \ref{TQR}, and finally the last by Lemma \ref{Riesz}.  We now turn to the proof of Lemma \ref{TQR}.

\begin{proof}
Let us ``slice'' the matrix $T\ci{Q,R}$ according to the ratio $\frac{s(Q)}{s(R)}$. Namely, let
$$
T^{(k)}_{Q,R}=\left\{\aligned T\ci{Q,R}&\qquad\text{if } s(Q)=\kappa^{-k}s(R)\,;\\0&\qquad\text{otherwise}\,,\endaligned\right.
$$
($k=0,1,2,\dots$).
To prove the lemma, it is enough to show that for every $k\ge 0$,
$$
\sum_{Q,R}T^{(k)}\ci{Q,R}a\ci Q b\ci R\le C\,2^{-\frac{\tau}{2}k}\Bigl[\sum_{Q}a\ci Q^2\Bigr]^{\frac12}\Bigl[\sum_{R}b\ci R^2\Bigr]^{\frac12}.
$$
The matrix $\{T^{(k)}_{Q,R}\}$ has a ``block'' structure since the variables $b\ci R$ corresponding to the cubes $R\in\mathcal{D}_2^{tr}$ for which $s(R)=\kappa^{j}$ can only interact with the variables $a\ci Q$ corresponding to the cubes $Q\in\mathcal{D}_1^{tr} $, for which $s(Q)=\kappa^{j-k}$. Thus, to get the desired inequality, it is enough to estimate each block separately, i.e., to demonstrate that
$$
\sum_{Q,R\,:\,s(Q)=\kappa^{j-k},s(R)=\kappa^j}T^{(k)}\ci{Q,R}a\ci Q b\ci R\leq C\,\Bigl[\sum_{Q\,:\,s(Q)=\kappa^{j-k}}a\ci Q^2\Bigr]^{\frac12}
\Bigl[\sum_{R\,:\,s(R)=\kappa^j}b\ci R^2\Bigr]^{\frac12}.
$$
Let us introduce the functions
$$
F(x):=\sum_{Q\,:\,s(Q)=\kappa^{j-k}}\frac{a\ci Q}{ \sqrt{\mu(Q)} }
\chi\ci Q(x)
\qquad
\text{and}
\qquad
G(x):=\sum_{R\,:\,\ell(R)=\kappa^{j}}\frac{b\ci R}{ \sqrt{\mu(R)} }
\chi\ci R(x).
$$
Note that the cubes of a given size in one dyadic lattice do not intersect (Property (ii) of Theorem \ref{dyadicmetric}),
and therefore at each point $x\in X$, at most one term in the sum can be
non-zero. Also observe that
$$
\|F\|\ci{L^2(X;\mu)}=
\Bigl[\sum_{Q\,:\,s(Q)=\kappa^{j-k}}a\ci Q^2\Bigr]^{\frac12}
\qquad
\text{and}
\qquad
\|G\|\ci{L^2(X;\mu)}=
\Bigl[\sum_{R\,:\,s(R)=\kappa^{j}}b\ci R^2\Bigr]^{\frac12}.
$$
Then the estimate we need can be rewritten as
$$
\int_X\int_X K_{j,k}(x,y)F(x)G(y)\, d\mu(x)\,d\mu(y) \le
C\,
\|F\|\ci{L^2(X;\mu)}\|G\|\ci{L^2(X;\mu)},
$$
where
$$
K_{j,k}(x,y)=\sum_{Q,R\,:\,s(Q)=\kappa^{j-k}, s(R)=\kappa^j}
\frac{s(Q)^{\frac{\tau}{2}}s(R)^{\frac{\tau}{2}}}{D(Q,R)^{m+\tau}}\chi\ci Q(x)\chi\ci
R(y).
$$
Again, for every pair of points $x,y\in X$, only one term in the sum can be
nonzero.
Since $\rho(x,y)+s(R)\le 3D(Q,R)$ for any $x\in Q$ and $y\in R$, we obtain
\begin{align*}
K_{j,k}(x,y)&=
C\,\kappa^{-\frac{\tau}{2}k}\frac{ s(R)^{\tau} }{ D(Q,R)^{m+\tau}}\\
&\lesssim \kappa^{-\frac{\tau}{2}k}
\frac{\kappa^{j\tau}}{[\kappa^j+\rho(x,y)]^{m+\tau}}
=:\kappa^{-\frac{\tau}{2}k} k_j(x,y).
\end{align*}
So, it is enough to check that
$$
\int_X\int_X k_{j}(x,y)F(x)G(y)\, d\mu(x)\,d\mu(y)\lesssim \|F\|\ci{L^2(X;\mu)}\|G\|\ci{L^2(X;\mu)}.
$$
We remind the reader that we called the balls ``non-Ahlfors balls" if
$$
\mu(B(x,r)) >r^m\,.
$$
According to the Schur test, it would suffice to prove that
for every $y\in X$, one has the estimate $\int_{X}k_j(x,y)\,d\mu(x)\lesssim 1$ and vice versa (i.e., for every $x\in X$, one has $\int_{X}k_j(x,y)\,d\mu(y)\lesssim 1$). Then the norm of the integral operator with kernel $k_j$ in $L^2(X;\mu)$ would be bounded by a constant and the proof of Lemma \ref{TQR} would be over.
If we assumed a priori that the supremum of radii of all  non-Ahlfors balls  centered at $y\in R$ with $s(R) =\kappa^j,$ were less than $\kappa^{j+1}$, then the needed estimate would be immediate. In fact, we can write
\begin{align*}
\int_{X}k_j(x,y)\,d\mu(x)&=\int_{B(y,\kappa^{j+1})}k_j(x,y)\,d\mu(x)+\int_{X\setminus B(y,\kappa^{j+1})}k_j(x,y)\,d\mu(x)\\&\lesssim \kappa^{-jm}\mu(B(y,\kappa^{j+1}))+\int_{X\setminus B(y,\kappa^{j+1})}\frac{\kappa^{j\tau}}{\rho(x,y)^{m+\tau}}\,d\mu(x)\\&
\lesssim \kappa^{-jm}\mu(B(y,\kappa^{j+1}))+\sum_{k=0}^\infty\frac{\kappa^{j\tau}}{(\kappa^k\kappa^{j+1})^{m+\tau}}\mu(B(y;\kappa^k\kappa^{j+1}))\\&
\lesssim \Bigl(1 + \sum_{k=0}^\infty\frac{1}{\kappa^{k\tau}}\Bigr)\approx 1.
\end{align*}
The passage from the second to the third line follows by exhausting the the space $X\setminus B(y,\kappa^{j+1})$ by ``annular regions'' and making obvious estimates using condition (\ref{nonhom}).

The difficulty with this approach is that we cannot guarantee the supremum of the radii of all non-Ahlfors balls  centered at $y$ be less than $\kappa^{j+1}$ for every $y\in X$.  Our measure may not have this uniform property.

So, generally speaking, we are unable to show that the integral operator with kernel $k_j(x,y)$ acts in $L^2(X;\mu)$;  but we {\it do not need} that much!  We only need to check that the corresponding bilinear form is bounded on two {\it given} functions $F$ and $G$. So, we are not interested in the points $y\in X$ for which $G(y)=0$ (or in the points $x\in X$, for which $F(x)=0$). But, by definition, $G$ can be non-zero only on the transit cubes in $\mathcal{D}_2$.   Here we used our convention that we omit in all sums the fact that $Q$ and $R$ are transit cubes, however they are!

Now let us notice that if (and this is the case for all $R$ in the sum we estimate in our lemma) $R\in\mathcal{D}_2^{tran}  $, then the supremum of radii of all non-Ahlfors balls  centered at $y\in R$ is bounded by $s(R)$ for every ${y\in R}$. Indeed, this is just Lemma \ref{obv2}. The same reasoning shows that if $Q\in\mathcal{D}_1^{tran}$, then the supremum of radii of all non-Ahlfors balls centered at $x\in Q$ is bounded by $\kappa^{j-k+1}\le \kappa^{j+1}$ whenever $F(x)\ne 0$, and we are done with Lemma \ref{TQR}. 
\end{proof}

Now, we hope, the reader will agree that the decision to declare the cubes contained in $\Omega$ terminal was a good one.  As a result, the fact that the measure $\mu$ is not Ahlfors did not put us in any real trouble -- we barely had a chance to notice this fact at all. But, it still remains to explain why we were so eager to have the extra condition
$$
|k(x,y)|\le \frac{1}{\max (d^m(x), d^m(y))}\,,\,\, d(x):= \dist (x, X\setminus \Omega)
$$
 on our Calder\'on--Zygmund kernel.  The answer is found in the next two sections.

\section{Short Range Interaction and Nonhomogeneous Paraproducts: Controlling Term $\sigma_3$.}
\label{shortrange}
  
Recall that the sum $\sigma_3$ is taken over the pairs $Q,R$, for which $s(Q)<\kappa^{-r} s(R)$ and $Q\cap R\ne\emptyset$. We would like to improve this condition to the demand that $Q$ lie ``deep inside'' one of the subcubes $R_j$.  Recall also that we defined the {\it skeleton} $sk\, R$ of the cube $R$ by
$$
sk\, R:=\bigcup_{j}\partial R_j.
$$
We have declared a cube $Q\in\mathcal{D}_1$ bad if there exists a cube $R\in\mathcal{D}_2$
 such that $s(R)>\kappa^r s(Q)$ and $\dist(Q, sk\, R)\le s(Q)^{\al}s(R)^{1-\al}$. Now, for every good cube $Q\in\mathcal{D}_1$, the conditions $s(Q)<\kappa^{-r} s(R)$ and $Q\cap R\ne \emptyset$
together imply that $Q$ lies inside one of the children $R_j$ of $R$. We will denote
this subcube by $R\ci Q$. The sum $\sigma_3$ can now be split into
$$
\sigma_3^{term}:=\sum_{ Q,R\,:\,Q\subset R,\, s(Q)<\kappa^{-r}  s(R),
                    \\ R\ci Q\text{ is terminal}}
(\Delta\ci Q f, T\Delta\ci R g)
$$
and
$$
\sigma_3^{tran}  :=\sum_{ Q,R\,:\,Q\subset R,\, s(Q)<\kappa^{-r}  s(R),
                    \\  R\ci Q\text{ is transit} }
(\Delta\ci Q f, T\Delta\ci R g).
$$

\smallskip

\subsection{ Estimation of $\sigma_3^{term}$.}
\label{sigmaterm}

First of all, write (recall that $R_j$ denote the children of $R$):
$$
\sigma_3^{term}=\sum_{j}\,\,
\sum_{ Q,R\,:\,s(Q)<\kappa^{-r}  s(R),
\\
Q\subset R_j\in\mathcal{D}_2^{term}}
(\Delta\ci Q f, T\Delta\ci R g).
$$
Clearly, it is enough to estimate the inner sum for every fixed, and so let us
do this for $j=1$. We have
$$
\sum_{ Q,R\,:\,s(Q)<\kappa^{-r}  s(R), \\ Q\subset
R_1\in\mathcal{D}_2^{term}}
(\Delta\ci Q f, T\Delta\ci R g)=
\sum_{R:R_1\in\mathcal{D}_2^{term}}\,\, \sum_{ Q:\,s(Q)<\kappa^{-r}  s(R), 
Q\subset R_1}
(\Delta\ci Q f, T\Delta\ci R g).
$$
Recall that the kernel $k$ of our operator $T$ satisfies
the estimate of Lemma \ref{obv1}
\begin{equation}
\label{konwholeR1}
|k(x,y)|  \lesssim \frac{1}{s(R)^m}\qquad\text{for all }
x\in R_1, y\in X.
\end{equation}
Hence,
\begin{equation}
\label{TonwholeR1}
|T\Delta\ci R g(x)|\lesssim\frac{\|\Delta\ci R g\|\ci{L^1(X;\mu)}}{s(R)^m}
\qquad\text{ for all }x\in R_1,
\end{equation}
and therefore
\begin{align*}
\|\chi\ci{R_1}\cdot T\Delta\ci R g \|\ci{L^2(X;\mu)}&\lesssim
\|\Delta\ci R g\|\ci{L^1(X;\mu)}\frac{\sqrt{\mu(R_1)}}{s(R)^m}\\
&\lesssim
\frac{{\mu(R)}}{s(R)^m}\|\Delta\ci R g\|\ci{L^2(X;\mu)}\le
A\,B\|\Delta\ci R g \|\ci{L^2(X;\mu)}.
\end{align*}
This follows because $\|\Delta\ci R g \|\ci{L^1(X:\mu)}
\le \sqrt{\mu(R)} \|\Delta\ci R g \|\ci{L^2(X;\mu)}$ and $\mu(R_1)\le \mu(R)$ hold trivially.  Additionally, by Lemma \ref{obv2} we have 
\begin{equation}
\label{transitagain}
\mu(R)\lesssim s(R)^m
\end{equation}
because $R$ (the father of the cube $R_1$) is a transit cube if $R_1$ is terminal.

Now, recalling Lemma \ref{Riesz}, and taking into account that
$\Delta\ci Q f\equiv 0$ outside $Q$, we get
\begin{align*}
\sum_{Q:\,Q\subset
R_1}&
|(\Delta\ci Q f, T\Delta\ci R g)|=
\sum_{Q:\,Q\subset
R_1}
|(\Delta\ci Q f, \chi\ci{R_1}\cdot T\Delta\ci R g)|\\
&\lesssim \|\chi\ci{R_1}\cdot T\Delta\ci R g \|\ci{L^2(X;\mu)}
\Bigl[\sum_{Q:\,Q\subset
R_1}\|\Delta\ci Q f
\|^2\ci{L^2(X;\mu)}\Bigr]^{\frac{1}{2}}\\
&\lesssim \|\Delta\ci R g \|\ci{L^2(X;\mu)}
\Bigl[\sum_{Q:\,Q\subset
R_1}\|\Delta\ci Q f
\|^2\ci{L^2(X;\mu)}\Bigr]^{\frac{1}{2}}.
\end{align*}
So, we obtain
$$
\sum_{R:\,R_1\in\mathcal{D}_2^{term}}\,
\sum_{Q:\,Q\subset R_1}
|(\Delta\ci Q f, T\Delta\ci R g)|
$$
$$\lesssim 
\sum_{R:\,R_1\in\mathcal{D}_2^{term}}
\|\Delta\ci R g \|\ci{L^2(X;\mu)}
\Bigl[\sum_{Q:\,Q\subset R_1}
\|\Delta\ci Q f
\|^2\ci{L^2(X;\mu)}\Bigr]^{\frac{1}{2}}
$$
$$
\lesssim
\Bigl[\sum_{R:\,R_1\in\mathcal{D}_2^{term}}
\|\Delta\ci R g \|^2\ci{L^2(X;\mu)}\Bigr]^{\frac12}
\Bigl[\sum_{R:\,R_1\in\mathcal{D}_2^{term}}\,\,
\sum_{Q:\,Q\subset R_1}
\|\Delta\ci Q f
\|^2\ci{L^2(X;\mu)}\Bigr]^{\frac{1}{2}}.
$$
But the terminal cubes in $\mathcal{D}_2$ do not intersect! Therefore every
$\Delta\ci Q f$ can appear at most once in the last double sum, and we get the bound 
\begin{multline*}
\sum_{R:\,R_1\in\mathcal{D}_2^{term}}\sum_{Q:\,Q\subset R_1}|(\Delta\ci Q f, T^*\Delta\ci R\psi)|\\\lesssim \Bigl[\sum_{R}\|\Delta\ci R g \|^2\ci{L^2(X;\mu)}\Bigr]^{\frac12}\Bigl[\sum_{Q}\|\Delta\ci Q f\|^2\ci{L^2(X;\mu)}\Bigr]^{\frac{1}{2}}\lesssim \|f\|\ci{L^2(X;\mu)}\|\psi\|\ci{L^2(X;\mu)}.
\end{multline*}
Lemma \ref{Riesz} has been used again in the last inequality.

\subsection{Estimation of $\sigma_3^{tran}$}
\label{s3transit}
 
 Recall that
 $$
 \sigma_3^{tran}  =\sum_{Q,R\,:\,Q\subset R,\, s(Q)<\kappa^{-r}  s(R),\\  R\ci Q\text{ is transit}}(\Delta\ci Q f, T^*\Delta\ci R g).
 $$
 Split every term in the sum as
 $$
 (\Delta\ci Q f, T\Delta\ci R\psi)=(\Delta\ci Q f, T(\chi\ci{R\ci Q}\Delta\ci R g))+(\Delta\ci Q f, T^*(\chi\ci{R\setminus R\ci Q}\Delta\ci R g)).
 $$
 Observe that since $Q$ is good, $Q\subset R$, and $s(Q)<\kappa^{-r}  s(R)$, we have
 $$
 \dist(Q,\supp \chi\ci{R\setminus R\ci Q}\Delta\ci R g)\ge\dist(Q,sk\, R)\ge s(Q)^{\al}s(R)^{1-\al}.
 $$
 Using Lemma \ref{FarInteractionLemma} and taking into account that the norm $\|\chi\ci{R\setminus R\ci Q}\Delta\ci R\psi\|\ci{L^2(X;\mu)}$ does not exceed $\|\Delta\ci R\psi\|\ci{L^2(X;\mu)}$, we conclude that the sum
 $$
 \sum_{Q,R\,:\,Q\subset R,\, s(Q)<\kappa^{-r}  s(R),   \\  R\ci Q\text{ is transit}}|(\Delta\ci Q f, T^*(\chi\ci{R\setminus R\ci Q}\Delta\ci R g))|
 $$
 can be estimated by the sum \eqref{twostar}.
 Thus, our task is to find a good bound for the sum
$$
\sum_{Q,R\,:\,Q\subset R,\, s(Q)<\kappa^{-r}  s(R),
                    \\  R\ci Q\text{ is transit}}
(\Delta\ci Q f, T^*(\chi\ci{R\ci Q}\Delta\ci R g)).
$$

Recalling the definition of $\Delta\ci R\psi$ and recalling that $R\ci Q$ is a
{\it transit\/} cube, we get
$$
\chi\ci{R\ci Q}\Delta\ci R g=c\ci{R,Q}\chi\ci{R\ci Q},
$$
where
$$
c\ci{R_Q}=\langle \psi\rangle\ci{R\ci Q}
-
\langle g\rangle\ci{R}
$$
is a {\it constant}.
So, our sum can be rewritten as
$$
\sum_{Q,R\,:\,Q\subset R,\, s(Q)<\kappa^{-r}  s(R),
                    \\  R\ci Q\text{ is transit}}
c\ci{R_Q}(\Delta\ci Q f, T^*(\chi\ci{R\ci Q})).
$$

Our next goal will be to extend the function $\chi\ci{R\ci Q}$ to the
function $1$ in every term.

Let us observe that
\begin{multline*}
(\Delta\ci Q f, T^*(\chi\ci{X\setminus R\ci Q}))=
\int_X\int_{X\setminus R\ci Q}k(x,y)\Delta\ci Q f(x) \,d\mu(x)\,d\mu(y)
\\=
\int_X\int_{X\setminus R\ci Q}
[k(x,y)-k(x\ci Q,y)]\Delta\ci Q f(x) \,d\mu(x)\,d\mu(y).
\end{multline*}

Note again that for every $x\in Q$, $y\in X\setminus R\ci Q$, we have
$$
\rho(x\ci Q,y)\ge \frac{s(Q)}{2}+\dist(Q,X\setminus R\ci Q)
\gtrsim s(Q)\gtrsim \rho(x,x\ci Q).
$$

Therefore,
$$
|k(x,y)-k(x\ci Q,y)|\lesssim \left(\frac{\rho(x,x\ci Q)}{\rho(x\ci Q,y)}\right)^\tau\frac{1}{\rho(x,y)^m}
\lesssim \frac{s(Q)^\tau}{\rho(x\ci Q,y)^{m+\tau}},
$$
and
$$
|(\Delta\ci Q f, T(\chi\ci{X\setminus R\ci Q}b))|\lesssim s(Q)^{\tau}\|\Delta\ci Q f\|\ci{L^1(X;\mu)}
\int_{X\setminus R\ci Q}
\frac{d\mu(y)}{\rho(x\ci Q,y)^{m+\tau}}.
$$
Now let us consider the sequence of cubes $R^{(j)}\in\mathcal{D}_2$, beginning with
$R^{(0)}=R\ci Q$ and gradually ascending ($R^{(j)}\subset R^{(j+1)}$,
$s(R^{(j+1)})=\kappa s(R^{(j)})$) to the starting cube $R^0=R^{(N)}$ of the
lattice $\mathcal{D}_2$. Clearly, all these cubes $R^{(j)}$ are transit cubes.

We have
$$
\int_{X\setminus R\ci Q}
\frac{d\mu(y)}{\rho(x\ci Q,y)^{m+\tau}}=
\int_{R^0\setminus R\ci Q}
\frac{d\mu(y)}{\rho(x\ci Q,y)^{m+\tau}}=
\sum_{j=1}^N
\int_{R^{(j)}\setminus R^{(j-1)}}
\frac{d\mu(y)}{\rho(x\ci Q,y)^{m+\tau}}\,.
$$
We call the $j$-th term of this sum $I_j$. Note now that, since $Q$ is good and $s(Q)<\kappa^{-r}  s(R)\le
\kappa^{-r}  s(R^{(j)})$ for all $j$, we have
$$
\dist(Q,R^{(j)}\setminus R^{(j-1)})\ge
\dist(Q, sk\, R^{(j)})\ge
s(Q)^{\al}s(R^{(j)})^{1-\al}.
$$
Hence
$$
I_j\le \frac{1}{[s(Q)^{\al}s(R^{(j)})^{1-\al}]^{m+\tau}}\int_{R^{(j)}} d\mu.
$$
Recalling that $\al=\frac{\tau}{2(m+\tau)}$, we see that the first factor equals
$$
\dfrac{1}{s(Q)^{\frac{\tau}{2}} s(R^{(j)})^{m+\frac{\tau}{2}}}\,.
$$
Since $R^{(j)}$ is transit, we have
$$
\int_{R^{(j)}}\,d\mu\lesssim \mu(R^{(j)})\lesssim s(R^{(j)})^m.
$$
Thus,
$$
I_j\lesssim {s(Q)^{\frac{\tau}{2}}s(R^{(j)})^{\frac{\tau}{2}}}=
\kappa^{-(j-1)\frac{\e}{2}}{s(Q)^{\frac{\tau}{2}}s(R)^{\frac{\tau}{2}}}.
$$
Summing over $j\ge 1$, we get
$$
\int_{X\setminus R\ci Q}
\frac{|b(y)|\,d\mu(y)}{\rho(x\ci Q,y)^{m+\tau}}= \sum_{j=1}^N I_j\lesssim {1-\kappa^{-\frac{\tau}{2}}}
\frac{1}{s(Q)^{\frac{\tau}{2}}s(R)^{\frac{\tau}{2}}}.
$$

Now let us note that
\begin{equation}
\label{cQR}
|c_{R_Q}| \le \frac{\|\Delta_R g\|_{L^1(R_Q,\mu)}}{\mu(R_Q)}\le \frac{\|\Delta_R g\|_{L^2(R_Q,\mu)}}{\sqrt{\mu(R_Q)}}\,.
\end{equation}

We finally obtain
\begin{multline*}
|(\Delta\ci Q f, T^*(\chi\ci{X\setminus R\ci Q}))|
\\
\lesssim
\frac{1}{\eta (1-\kappa^{-\frac{\tau}{2}})}
\left[\frac{s(Q)}{s(R)}\right]^{\frac{\tau}{2}}
\sqrt{\frac{\mu(Q)}{\mu(R\ci Q)}}
\|\Delta\ci Q f\|\ci{L^2(X;\mu)}
\|\Delta\ci R g\|\ci{L^2(X;\mu)}
\end{multline*}
and
\begin{multline*}
\sum_{Q,R\,:\,Q\subset R,\, s(Q)<\kappa^{-r}  s(R), R\ci Q\text{is transit}}
|c\ci{R,Q}|\cdot|(\Delta\ci Q f, T^*(\chi\ci{X\setminus
R\ci Q}))|
\\
\lesssim
\frac{1}{\eta (1-\kappa^{-\frac{\tau}{2}})}
\sum_{j}\,\sum_{Q,R\,:\,Q\subset R_j}
\left[\frac{s(Q)}{s(R)}\right]^{\frac{\tau}{2}}
\sqrt{\frac{\mu(Q)}{\mu(R_j)}}
\|\Delta\ci Q f\|\ci{L^2(X;\mu)}
\|\Delta\ci R g\|\ci{L^2(X;\mu)}.
\end{multline*}

\begin{lemma}
\label{TQRshortrange}
For every two
families  $\{a\ci Q\}\ci{Q\in\mathcal{D}_1^{tr}  }$ and $\{b\ci R\}\ci{R\in\mathcal{D}_2^{tr}  }$
of nonnegative numbers, one has
$$
\sum_{Q,R:Q\subset R_1}T\ci{Q,R}a\ci Q b\ci R\le
\frac{1}{ 1-\kappa^{-\frac{\tau}{2}} }
\Bigl[\sum_{Q}a\ci Q^2\Bigr]^{\frac12}
\Bigl[\sum_{R}b\ci R^2\Bigr]^{\frac12}.
$$
\end{lemma}

\begin{proof}
Let us ``slice'' the matrix $T\ci{Q,R}$ according to the ratio
$\frac{s(Q)}{s(R)}$. Namely, let
$$
T^{(k)}_{Q,R}=\left\{
\aligned
T\ci{Q,R},&\qquad\text{if } Q\subset R_1,\  s(Q)=\kappa^{-k}s(R);
\\
0,&\qquad\text{otherwise}
\endaligned
\right.
$$
($k=1,2,\dots$).
It is enough to show that for every $k\ge 0$,
$$
\sum_{Q,R}T^{(k)}\ci{Q,R}a\ci Q b\ci R\le
\kappa^{-\frac{\tau}{2} k }
\Bigl[\sum_{Q}a\ci Q^2\Bigr]^{\frac12}
\Bigl[\sum_{R}b\ci R^2\Bigr]^{\frac12}.
$$
The matrix $\{T^{(k)}_{Q,R}\}$ has a very good ``block'' structure:
every $a\ci Q$ can interact with {\it only one} $b\ci R$.
So, it is enough to estimate each block separately, i.e., to show that
for every fixed $R\in\mathcal{D}_2^{tran}  $,
$$
\sum_{Q:\,Q\subset R_1,\, \ell(Q)=\kappa^{-k}\ell(R)}
\kappa^{-\frac{\tau}{2} k }
\sqrt{\frac{\mu(Q)}{\mu(R_1)}}
a\ci Q b\ci R\le
\kappa^{-\frac{\tau}{2} k }
\Bigl[\sum_{Q}a\ci Q^2\Bigr]^{\frac12}
b\ci R.
$$
But, reducing both parts by the non-essential factor
$\kappa^{-\frac{\tau}{2} k }b\ci R$, we see that this estimate is equivalent to the trivial
estimate
\begin{multline*}
\!\!\sum_{Q:\,Q\subset R_1,\, s(Q)=\kappa^{-k}s(R)}
\sqrt{\frac{\mu(Q)}{\mu(R_1)}}
a\ci Q \\
\leq
\Bigl[
\sum_{Q:\,Q\subset R_1,\, s(Q)=\kappa^{-k}s(R)}
\frac{\mu(Q)}{\mu(R_1)}\Bigr]^{\frac12}
\Bigl[\sum_{Q}a\ci Q^2\Bigr]^{\frac12}\leq
\Bigl[\sum_{Q}a\ci Q^2\Bigr]^{\frac12},
\end{multline*}
(since cubes $Q\in \mathcal{D}_1$  of fixed size do not intersect,
$
\sum_{Q:\,Q\subset R_1,\, s(Q)=\kappa^{-k}s(R)}
\mu(Q)\le \mu(R_1)
$\,).

\end{proof} 

\begin{remark}
We did not use here the fact that $\{a_Q\},\{b_R\}$ are supported on transit cubes. We actually proved
\begin{lemma}
\label{actually}
The matrix $\{T\ci{Q,R}\}$ defined
by
$$
T\ci{Q,R}:=\left[\frac{s(Q)}{s(R)}\right]^{\frac{\tau}{2}}
\sqrt{\frac{\mu(Q)}{\mu(R_1)}}\qquad\quad(Q\subset R_1),
$$
generates a bounded operator in $l^2$.
\end{lemma}
\end{remark}

We just finished estimating an extra term which appeared when we extend $\chi\ci{R\ci Q}$ to the whole $1$. So, the extension of $\chi\ci{R\ci Q}$ to the function $1$ does not cause much harm, and we are left with estimating the sum
$$
\sum_{Q,R\,:\,Q\subset R,\, s(Q)<\kappa^{-r}  s(R),
                    \\  R\ci Q\text{ is transit}}
c\ci{R_Q}(\Delta\ci Q f, T^* 1).
$$
Note that the inner product $(\Delta\ci Q f, T^* 1)$ {\it does not depend\/} on $R$ at all, so it seems to be a good idea to sum over $R$ for fixed $Q$ first.

Recalling that

$$
c\ci{R_Q}=\langle  g\rangle\ci{R\ci Q}-\langle g\rangle\ci{R}
$$
and that $\Lambda\psi=0 \Longleftrightarrow \langle \psi\rangle\ci{R^0}=0$, we
conclude that for every $Q\in\mathcal{D}_1^{tran}  $ that really appears in the above sum,
$$
\sum_{R\,:\,R\supset Q,\, s(R)>\kappa^{m}s(Q),                    \\  R\ci Q\text{ is transit}}c\ci{R_Q}=\langle g\rangle\ci{R_Q}\,.
$$

\vspace{.1in}

\noindent{\bf Definition.}
Let $R(Q)$ be the smallest {\it transit} cube $R\in \mathcal{D}_2$ containing $Q$
and such that $s(R)\ge \kappa^r s(Q)$.

\vspace{.1in}

So, we obtain the sum

$$
\sum_{Q:\,s(Q)<\kappa^{-r}  s(R)} \langle g\rangle\ci{R(Q)}(\Delta\ci Q f, T^* 1)
$$
to take care of.

\vspace{.1in}

\noindent{\bf Remark.} Let us recall that we had the convention that  says that the cubes $Q$ considered are only good ones (and of course they are only transit cubes). The range of summation should be $Q\in \mathcal{D}_1^{tran} $, $Q$ is good (default); there exists a cube $R\in\mathcal{D}_2^{tran} $ such that $s(Q)<\kappa^{-r} s(R)$, $Q\subset R$ and the child $R\ci Q$ (the one containing $Q$) of $R$  is transit. In other words, in fact, the sum is written formally incorrectly. We have to replace $R(Q)$ by $R_Q$ in the summation. However, the smallest transit cube containing $Q$ (this is $R(Q)$) and the smallest transit child (containing $Q$) of a certain subcube $R$ of $R^0$  (this child is $R_Q$) are of course the same cube, unless $R(Q)=R^0$. Thus the sum formally has some extra terms corresponding to $R(Q) =R^0$. But, they all are zeros!  In one of the first reductions, we were allowed to work only with with $g$ such that $\Lambda g=0$ (recall that $\Lambda g$ means the average of $g$ with respect to $\mu$), so $\langle g\rangle_{R(Q)} =0$ if $R(Q)=R^0$.

\subsection{Pseudo-$BMO$ and special paraproduct}
\label{bmo}

To introduce the paraproduct operator, we rewrite our sum as follows
\begin{align*}
\sum_{Q:\,s(Q)<\kappa^{-r}  s(R)}\langle g\rangle\ci{R(Q)}(\Delta\ci Q f, T^*1)&= \sum_{Q:\,s(Q)<\kappa^{-r}  s(R)}\langle g\rangle\ci{R(Q)}( f,\Delta_Q^* T^*1) \\&= \left( f, \sum_{Q:\,s(Q)<\kappa^{-r}  s(R)}\langle g\rangle\ci{R(Q)}\Delta_Q T^*1\right)\,.
\end{align*}

We use the fact that $\Delta_Q^* =\Delta_Q$.  We now introduce the paraproduct operator, which will allow us to control term $\sigma_3^{tran}$.

\begin{defi}
Given a function $F$, the \textit{paraproduct with symbol} $F$ is the function
$$
\Pi_{F}g(x) := \sum_{R\in \mathcal{D}_2,\,R\subset R^0}\langle g\rangle_R\sum_{Q\in\mathcal{D}_1,\,Q\,\text{good and transit},\,s(Q)= \kappa^{-r}  s(R)} \Delta_Q F(x)\,.
$$
\end{defi}

As in the case when the metric space is $\mathbb{R}^d$, the behavior of the paraproduct operators will be governed by ``BMO'' conditions on the symbol $F$.  In the case of a metric space though, we face an additional wrinkle since we have to overcome the challenge of dealing with the dyadic cubes, and we need an appropriate notion of ``dilation'' in the metric space.

Recall that we defined the dilation by the parameter $\lambda\geq 1$ of a set $S\subset X$ by
$$
\lambda\cdot S:=\{x\in X: \textnormal{dist}(x,S)\leq(\lambda-1)\textnormal{diam}{S}\}
$$
Note that $S\subset\lambda\cdot S$.

\begin{defi}
A function $F\in L^2(X;\mu)$ will be called a ``pseudo-$BMO$ function'' if there exists $\La>1$ such that for any cube $Q$ with $\mu(sQ)\le K\, s^m\textnormal{diam}(Q)^m$, $s\ge 1$, we have
$$
\int_Q |F(x)-\langle F\rangle_Q|^2\,d\mu(x) \le C\, \mu(\La\,Q)\,.
$$
\end{defi}

 \begin{lemma}
\label{Tstar1}
Let $\mu$, $T$ satisfy the assumptins of Theorem \ref{md}. Then
\begin{equation}
\label{para}
T^*1 \in \textnormal{pseudo-BMO}\,.
\end{equation}
Here $C$ depends only on the constants of  Theorem \ref{md}.
\end{lemma}

\begin{proof}
For $x\in Q$ we write $T^*1(x) = (T^*\chi_{\La Q})(x) + (T^*\chi_{X\setminus\La Q})(x)=: \vf(x) +\psi(x)$.
First, we notice that 
$$
x,y \in Q\Rightarrow |\psi(x)-\psi(y)| \le C(K,\La)\,,
$$ where $K$ is the constant form our definition above. This is easy:
$$
|\psi(x)-\psi(y)|\le \int_{X\setminus\Lambda Q} |k(x,t)-k(y,t)| \,d\mu(t)= \sum_{j=1}^{\infty} \int_{\La^{j+1} Q\setminus \La^j Q} |k(x,t)-k(y,t)| \,d\mu(t) \le
$$
$$
 \sum_{j=1}^{\infty} \frac{\textnormal{diam}(Q)^{\tau}}{(\La^j\textnormal{diam}(Q))^{m+\tau}}\, K (\La^j\textnormal{diam}(Q))^{m}=\sum_{j=1}^\infty\frac{K}{\Lambda^{j\tau}}\le C(K,\La, \tau)\,.
$$

Therefore,
$$
\int_Q |\psi(x)-\langle \psi\rangle_Q|^2\,d\mu(x) \lesssim \mu(Q)\leq\mu(\Lambda Q)\,.
$$
But, 
$$
\int_Q |\vf(x)-\langle \vf\rangle_Q|^2\,d\mu(x) \lesssim \int_Q |T^*\chi_{\La Q}|^2\,d\mu \le A\,\mu(\La Q)
$$
by the $T1$ assumption of Theorem \ref{md}.

\end{proof}

\begin{lemma}
\label{paraproductlm}
Let $\mu$, $T$ satisfy the assumptins of Theorem \ref{md}. Then
\begin{equation}
\label{paraproduct}
\|\Pi_{T^*1}\|_{L^2(X;\mu)\to L^2(X;\mu)} \leq C\,.
\end{equation}
Here $C$ depends only on the constants of  Theorem \ref{md}.
\end{lemma}

\begin{proof}

Let $F=T^*1$. In the definition of $\Pi_F$ all $\Delta_Q$ are mutually orthogonal. So it is easy to see that
$$
\|\Pi_F g\|_{L^2(X;\mu)}^2 = \sum_{R\in \mathcal{D}_2,\,R\subset R^0}|\langle g\rangle_R|^2\sum_{Q\in\mathcal{D}_1,\,Q\,\text{good and transit},\,s(Q)= \kappa^{-r}  s(R)} \|\Delta_Q F\|_{L^2(X;\mu)}\,.
$$

Put 
$$
a_R:= \sum_{Q\in\mathcal{D}_1,\,Q\,\text{good and transit},\,s(Q)= \kappa^{-r}  s(R)} \|\Delta_Q F\|_{L^2(X;\mu)}\,.
$$

By Carleson Embedding Theorem, it is enough to prove that for every $S\in \Dk_2$

\begin{equation}
\label{cet}
\sum_{R\in \Dk_2, R\subset S} a_R \le C\,\mu(S)\,.
\end{equation}

This is the same as 

\begin{equation}
\label{cet1}
\sum_{Q\in\mathcal{D}_1,\,Q\,\text{transit},\,s(Q)\le \kappa^{-r}  s(R), \dist(Q, \pd R) \ge s(Q)^{\alpha}s(R)^{1-\alpha}} \|\Delta_Q F\|_{L^2(X;\mu)} \le C\, \mu(R)\,.
\end{equation}

Let us consider a Whitney decomposition of $R$ into disjoint cubes $P$, such that $1.5 P \subset R$, $1.4 P$ have only bounded multiplicity $C(d)$ of intersection.  This can be accomplished by modifying the arguments found in Section 7 of \cite{HM}.

Consider the sums
\begin{equation}
\label{cet2}
s_P:=\sum_{Q\in\mathcal{D}_1,\,Q\,\text{transit},\,s(Q)\le \kappa^{-r}  s(R), Q\cup P\neq \emptyset, \dist(Q, \pd R) \ge s(Q)^{\alpha}s(R)^{1-\alpha}} \|\Delta_Q F\|_{L^2(X;\mu)}\,.
\end{equation}

This $s_P$ can be zero if there is no transit cubes as above intersecting it. But if $s_P\neq 0$ then necessarily
$$
\mu(P)\le A(d)s(P)^m\,,
$$ 
and moreover
$$
\mu(sP)\le A(d)\,s^m\,s(P)^m\,,\,\,\forall s\ge 1\,.
$$

In fact, in this case $P$ intersects a transit cube $Q$, which by elementary geometry is ``smaller"'\ than $P$: $s(Q) \le c(r,d) s(P)$. But then the above inequalities follow from the definition of {\it transit}.

It is also clear that for large $r$ and for $Q,P$ as above
$$ 
Q\cap P\neq \emptyset\Rightarrow Q\subset 1.2 \,P\,.
$$

Therefore,
$$
s_P\neq 0\Rightarrow s_P \le \sum_{Q\in\mathcal{D}_1,\,Q\,\text{transit},\,s(Q)\le \kappa^{-r}  s(R), Q\subset 1.2\, P\, \dist(Q, \pd R) \ge s(Q)^{\alpha}s(R)^{1-\alpha}} \|\Delta_Q F\|_{L^2(X;\mu)}\,.
$$
So

$$
s_P\neq 0\Rightarrow s_P \le \int_{1.2P} |F-\langle F\rangle_{1.2\,P}|^2\,d\mu\le C\,\mu(1.4\, P)\,.
$$
The last inequality follows from Lemma
\ref{Tstar1}.

Now we add all $s_P$'s. We get $\le C\, \sum \mu (1.4\,P)$. This is smaller than $C_1\, \mu(R)$ as $1.4P$'s have multiplicity $C(d)<\infty$.
\end{proof}

\section{The Diagonal Sum: Controlling Term $\sigma_1$.}
\label{diag}

To complete the estimate of $|(Tf_{good},g_{good})|$ in only remains to estimate $\sigma_1$. But notice that
$$
\|\Delta_Q f\|_{L^1(X;\mu)} \le \|\Delta_Q f\|_{L^2(X;\mu)} \sqrt{ \mu(Q)}\textnormal{ and } \|\Delta_R g\|_{L^1(X;\mu)} \le \|\Delta_R g\|_{L^2(X;\mu)} \sqrt{ \mu(R)}\,.
$$ 

Remember that all cubes $Q$ and $R$ in the sums considered at this point are transit cubes.  In particular, in $\sigma_1$ we have that $Q$ and $R$ are close and of the almost same size. If a son of $Q$, $S(Q)$, is terminal, then by Lemma \ref{obv1}
$$
|(T\chi_{S(Q)}\Delta_Q f, \Delta_R g)|\le \frac{\sqrt{ \mu(Q)}\sqrt{ \mu(R)}}{s(Q)^m} \|\Delta_Q f\|_{L^2(X;\mu)}  \|\Delta_R g\|_{L^2(X;\mu)} \,.
$$
The sons are terminal, but $Q$ and $R$ are transit, so $\mu(Q) \lesssim s(Q)^m\approx s(R)^m$.
Summing such pairs (and symmetric ones, where a son of $R$ is terminal) we get $C(r)\|f\|_{L^2(X;\mu)}\|g\|_{L^2(X;\mu)}$.

We are left with the part of $\sigma_1$, where we sum over $Q$ and $R$ such that their sons are transit. Then we use pairing
$$
|(T\chi_{S(Q)}\Delta_Q f, \chi_S(R)\Delta_R g)|\le |c_{S(Q)}||c_{S(R)}|\sqrt{\mu(S(Q))\mu(S(R))}\,.
$$ 
The estimate above follows from our $T1$ assumption in Theorem \ref{md}.
Now using \eqref{cQR}, again we obtain
$$
|(T\chi_{S(Q)}\Delta_Q f, \chi_{S(R)}\Delta_R g)|\le C\, \|\Delta_Q f\|_{L^2(X;\mu)}  \|\Delta_R g\|_{L^2(X;\mu)} \,.
$$

This completes the proof of Theorem \ref{md}.


\begin{bibdiv}
\begin{biblist}



\bib{C}{article}{
   author={Christ, Michael},
   title={A $T(b)$ theorem with remarks on analytic capacity and the Cauchy
   integral},
   journal={Colloq. Math.},
   volume={60/61},
   date={1990},
   number={2},
   pages={601--628}
}

	
\bib{HM}{article}{
   author={Hyt\"onen, T.},
   author={Martikainen, H.},
   title={Non-Homogeneous Tb Theorem on Metric Spaces}
   pages={preprint}
}

\bib{NTV1}{article}{
   author={Nazarov, F.},
   author={Treil, S.},
   author={Volberg, A.},
   title={The $Tb$-theorem on non-homogeneous spaces},
   journal={Acta Math.},
   volume={190},
   date={2003},
   number={2},
   pages={151--239}
}
		

\bib{NTV2}{article}{
   author={Nazarov, F.},
   author={Treil, S.},
   author={Volberg, A.},
   title={Accretive system $Tb$-theorems on nonhomogeneous spaces},
   journal={Duke Math. J.},
   volume={113},
   date={2002},
   number={2},
   pages={259--312}
}
		

\bib{NTV3}{article}{
   author={Nazarov, F.},
   author={Treil, S.},
   author={Volberg, A.},
   title={Weak type estimates and Cotlar inequalities for Calder\'on-Zygmund
   operators on nonhomogeneous spaces},
   journal={Internat. Math. Res. Notices},
   date={1998},
   number={9},
   pages={463--487}
}
		
\bib{NTV4}{article}{
   author={Nazarov, F.},
   author={Treil, S.},
   author={Volberg, A.},
   title={Cauchy integral and Calder\'on-Zygmund operators on nonhomogeneous
   spaces},
   journal={Internat. Math. Res. Notices},
   date={1997},
   number={15},
   pages={703--726}
}

\bib{SW}{article}{
   author={Sawyer, E.},
   author={Wheeden, R. L.},
   title={Weighted inequalities for fractional integrals on Euclidean and
   homogeneous spaces},
   journal={Amer. J. Math.},
   volume={114},
   date={1992},
   number={4},
   pages={813--874}
}


		






		
\bib{V}{book}{
   author={Volberg, A.},
   title={Calder\'on-Zygmund capacities and operators on nonhomogeneous
   spaces},
   series={CBMS Regional Conference Series in Mathematics},
   volume={100},
   publisher={Published for the Conference Board of the Mathematical
   Sciences, Washington, DC},
   date={2003},
   pages={iv+167},
   isbn={0-8218-3252-2}
}

\bib{VW}{article}{
  author={Volberg, A.},
  author={Wick, B. D.},
  title={Bergman-type Singular Operators and the Characterization of Carleson Measures for Besov--Sobolev Spaces on the Complex Ball},
  date={2009},
  pages={preprint}
}


		

\end{biblist}
\end{bibdiv}
\end{document}